\documentclass[12pt]{amsart}
\usepackage[dvips]{graphicx}
\usepackage{amsmath,graphics}
\usepackage{amsfonts,amssymb}
\usepackage{xypic}
\usepackage{comment}
\usepackage{color}
\specialcomment{proofc}{}{}
\includecomment{proofc}
\theoremstyle{plain}
\newtheorem*{theorem*}{Theorem}
\newtheorem*{lemma*} {Lemma}
\newtheorem*{corollary*} {Corollary}
\newtheorem*{proposition*}{Proposition}
\newtheorem*{conjecture*}{Conjecture}
\newtheorem{theorem}{Theorem}[section]
\newtheorem{lemma}[theorem]{Lemma}
\newtheorem*{theorem1*}{Theorem 1}
\newtheorem*{theorem2*}{Theorem 2}
\newtheorem*{theorem3*}{Theorem 3}

\newtheorem{proposition}[theorem]{Proposition}

\theoremstyle{remark}
\newtheorem*{remark}{Remark}

\newtheorem{example*}{Example}

\theoremstyle{definition}

\def\gl{\mbox{GL}} \def\Q{\Bbb{Q}}  \def\Z{\Bbb{Z}} \def\R{\Bbb{R}} 
\def\N{\Bbb{N}}  \def\l{\lambda}  
 \def\a{\alpha} \def\g{\gamma}  \def\bp{\begin{pmatrix}}
\def\sm{\setminus} \def\ep{\end{pmatrix}} \def\bn{\begin{enumerate}} 
  \def\div{\mbox{div}} \def\en{\end{enumerate}}
\def\ba{\begin{array}} \def\ea{\end{array}}  
   \def\a{\alpha} \def\b{\beta}

\def\ker{\mbox{Ker}}\def\be{\begin{equation}} \def\ee{\end{equation}} 
   
 \def\hom{\mbox{Hom}}  
 \def\aut{\mbox{Aut}}  
 \def\dim{\mbox{dim}}  \def\zgt{\Z[G][t^{\pm 1}]}

\def\zt{\Z[t^{\pm 1}]} \def\qt{\Q[t^{\pm 1}]}   
 \def\rt{R[t^{\pm 1}]}

\def\w{\omega}
\def\spin{\operatorname{spin}}

\def\tpm {[t^{\pm 1}]}

\def\MN{\mathcal{MN}}
\def\rnt{R^n\tpm}

\begin{document}
\title[A vanishing theorem for twisted Alexander polynomials]{A vanishing theorem for twisted Alexander polynomials with applications to symplectic 4-manifolds}
\author{Stefan Friedl}
\address{Mathematisches Institut\\ Universit\"at zu K\"oln\\   Germany}
\email{sfriedl@gmail.com}

\author{Stefano Vidussi}
\address{Department of Mathematics, University of California,
Riverside, CA 92521, USA} \email{svidussi@math.ucr.edu} \thanks{S. Vidussi was partially supported by NSF grant
\#0906281.}
\date{\today}

\begin{abstract}
In this paper we show that given any 3--manifold $N$ and any non--fibered class in $H^1(N;\Z)$ there exists a representation such that the corresponding twisted Alexander polynomial is zero.
We obtain this result by extending earlier work of ours and by combining this  with recent results of Agol and Wise on separability of $3$--manifold groups.
This result allows us to
completely classify symplectic $4$--manifolds with a free circle action, and to determine their symplectic cones.
\end{abstract}
\maketitle

\section{Introduction and main results}

A \emph{$3$--manifold pair} is a pair $(N,\phi)$ where  $N$ is a compact, orientable, connected 3--manifold with toroidal or empty boundary,
and $\phi\in H^1(N;\Z)=\hom(\pi_1(N),\Z)$ is a nontrivial class.
We say that a 3--manifold pair
\emph{$(N,\phi)$ fibers over $S^{1}$}
 if there exists
 a fibration $p\colon N\to S^1$ such that the induced map $p_*\colon \pi_1(N)\to \pi_1(S^1)=\Z$ coincides with $\phi$.
 We refer to such $\phi$ as a \emph{fibered class}.

Given a 3-manifold pair  $(N,\phi)$ and an epimorphism  $\a\colon \pi_1(N)\to G$ onto a finite group we can consider the twisted Alexander polynomial $\Delta_{N,\phi}^\a\in \zt$, whose definition is summarized in Section \ref{sectionufd}. It is well--known that the twisted Alexander polynomials of a fibered class $\phi \in H^{1}(N)$ are monic and that their degrees  are determined by the Thurston norm.
In \cite{FV11a} the authors showed that this condition is in fact sufficient to determine fiberedness.
More precisely,  if a nontrivial class $\phi \in H^1(N)$ is not fibered, then we proved in \cite{FV11a} that there exists a twisted Alexander polynomial  $\Delta_{N,\phi}^\a$ that fails to be monic or to have correct degree.
We refer to Theorem \ref{thm:fv08} for the exact statement.

In previous work (see \cite{FV08b}) the authors discussed how a stronger  result, namely the vanishing of some twisted Alexander polynomial  $\Delta_{N,\phi}^\a$,  would follow assuming appropriate separability conditions for the fundamental group of $N$. In this paper, building on the techniques of \cite{FV08b} supplemented with ideas of Wilton--Zalesskii \cite{WZ10}, and using new results on separability for $3$--manifolds groups due to Agol and Wise we  extend the  vanishing result to all $3$--manifolds. More precisely, we have the following

\begin{theorem}\label{mainthmfibintro}
Let $(N,\phi)$ be a 3--manifold pair.
If $\phi \in H^1(N)$ is nonfibered, there exists an epimorphism $\a\colon \pi_1(N)\to G$ onto a finite group $G$ such that
\[ \Delta_{N,\phi}^\a =0.\]
\end{theorem}

In brief, the strategy to prove Theorem \ref{mainthmfibintro} is the following. The techniques of  \cite{FV08b} allow to prove a vanishing result under certain subgroup separability properties for the fundamental group of $N$. Recall that  a group $\pi$ is called \emph{locally extended residually finite}  (\emph{LERF} for short, or subgroup separable) if for any finitely generated subgroup $A\subset \pi$ and any $g\in \pi\sm A$ there exists an epimorphism $\a\colon \pi\to G$ to a finite group $G$ such that $\a(g)\not\in \a(A)$.
Using work of Wilton  and Zalesskii \cite{WZ10}, we will reduce the separability condition to the hyperbolic pieces of $N$.
It has been a long standing conjecture that fundamental groups of hyperbolic 3--manifolds are LERF.
Recently Dani Wise has made remarkable progress  towards an affirmative answer. The following theorem combines the statements
of Corollaries 14.3 and 14.16 and Theorem 16.28 of Wise \cite{Wi12}. (We refer to \cite{Wi09,Wi11,Wi12} and \cite{AFW12} for background material, definitions and further information.)

 \begin{theorem}\label{thm:wiseintro}\textbf{\emph{(Wise)}}
If $N$ is either a closed hyperbolic 3-manifold which admits a geometrically finite surface or if $N$ is  a hyperbolic 3-manifold with nontrivial boundary, then $\pi_1(N)$ is  LERF.
\end{theorem}
In Section \ref{sec:proof} we will show how this result can be used to complete the proof of Theorem \ref{mainthmfibintro}.
We want to add that Agol \cite{Ag12} has made further progress in this direction, showing that the fundamental groups of \textit{any} closed hyperbolic 3-manifold is LERF.
(This implies as well, by work of  Manning and Martinez-Pedroza \cite[Proposition~5.1]{MMP10} that the fundamental group  of any hyperbolic 3-manifold with toroidal  boundary is LERF.)

We now discuss the applications of Theorem \ref{mainthmfibintro}.
In Section \ref{section:firstapps} we will see that the combination of Theorem \ref{mainthmfibintro} with work of Goda and Pajitnov \cite{GP05} implies a result on Morse-Novikov numbers of multiples of a given knot.
Furthermore, we will see that the combination of Theorem \ref{mainthmfibintro} with work of  Silver and Williams \cite{SW09b}
gives rise to  a fibering criterion in terms of the number of finite covers of the $\phi$-cover of $N$.
Arguably, however, the most interesting application of Theorem \ref{mainthmfibintro} is contained in Section \ref{section:symp}, and regards the study of closed 4-manifolds with a free circle action which admit a symplectic structure.
The main result of Section \ref{section:symp} is then the proof of the `(1) implies (3)' part of the following theorem:

\begin{theorem}\label{thm:symp}
Let $N$ be a closed 3-manifold and let $p\colon M\to N$ be an $S^1$--bundle over $N$. We denote by $p_*\colon H^2(M;\R)\to H^1(N;\R)$ the map that is given by integration along the fiber.
Let $\psi\in H^2(M;\R)$. Then the following are equivalent:
\bn
\item $\psi$ can be represented by a symplectic structure,
\item $\psi$ can be represented by a symplectic structure that is $S^1$--invariant,
\item $\psi^2>0$ and  $\phi=p_*(\psi) \in H^1(N;\R)$  lies in the open cone on a fibered face of the Thurston norm ball of $N$.
\en
\end{theorem}

(See Section \ref{section:symp} for details on the other implications.)

The implication `(1) implies (3)' had already been shown to hold in the following cases:
 \bn
 \item for reducible 3--manifolds by McCarthy \cite{McC01},
 \item if $N$ has vanishing Thurston norm by Bowden \cite{Bow09} and \cite{FV08c}, or if $N$  is a graph manifold, \cite{FV08c},
 \item if the canonical class of the symplectic structure is trivial, \cite{FV11c},
 \item if  $M$ is the trivial $S^1$-bundle over $N$, i.e. the case that $M=S^1\times N$, see \cite{FV11a} for details.
 \en

\begin{remark}
\bn
\item
This paper can be viewed as the (presumably) last paper in a long sequence of papers
\cite{FV08a,FV08b,FV08c,FV11a,FV11b,FV11c} by the authors on twisted Alexander polynomials, fibered 3-manifolds and symplectic structures.
\item
Some steps in the proof of Theorem \ref{thm:symp} (notably Propositions \ref{prop:pullback} and \ref{monicness}) already appeared in an unpublished
manuscript by the authors (see \cite{FV08c}).
\item  Bowden \cite{Bow12} used Theorem \ref{thm:symp} to determine which 4-manifolds with a fixed point free circle action are symplectic.
\en
\end{remark}

 \subsection*{Acknowledgment.} We wish to thank Henry Wilton for very helpful conversations and we are grateful to the referee for thoughtful comments.

\subsection*{Convention.} Unless it says specifically otherwise, all groups are assumed to be finitely generated, all manifolds are assumed to be orientable, connected and compact, and all 3-manifolds are assumed to have empty or toroidal boundary.

\section{Definition  of twisted Alexander polynomials} \label{section:definition}\label{sectionufd}

In this section we quickly recall the definition of twisted Alexander polynomials.
This invariant was initially introduced by
 Lin \cite{Li01}, Wada \cite{Wa94} and Kirk--Livingston \cite{KL99}. We refer to \cite{FV10a} for a detailed presentation.

Let $X$ be a finite CW complex, let $\phi\in H^1(X;\Z)=\hom(\pi_1(X),\Z)$  and let $\a\colon \pi_1(X)\to \gl(n,R)$ be a representation over  a Noetherian unique factorization domain $R$. In our applications we will take $R=\Z$ or $R=\Q$.
 We can now define a left $\Z[\pi_1(X)]$--module
structure on $R^n\otimes_\Z \zt=:\rnt$ as follows:
\[  g\cdot (v\otimes p):=  (\a(g) \cdot v)\otimes (t^{\phi(g)}p), \]
where $g\in \pi_1(X), v\otimes p \in R^n\otimes_\Z \zt = \rnt$.
Put differently, we get a
representation $\a\otimes \phi\colon \pi_1(X)\to \gl(n,\rt)$.

Denote by $ \widetilde{X}$ the universal cover of $X$.
Letting $\pi=\pi_1(X)$, we use
the representation $\a\otimes \phi$ to regard $\rnt$ as a left $\Z[\pi]$--module.
The chain complex $C_*(\widetilde{X})$ is also a left $\Z[\pi]$--module via deck transformations.
Using the natural involution $g\mapsto g^{-1}$ on the group ring $\Z[\pi]$ we can  view $C_*(\widetilde{X})$ as a right $\Z[\pi]$--module.
 We can therefore consider
the  tensor products
\[ C_*^{\phi\otimes \a}(X;\rnt):=C_*(\tilde{X})\otimes_{\Z[\pi_1(X)]}\rnt,\]
which form a complex of $\rt$-modules.
We then consider the $\rt$--modules
$$
H_*^{\phi\otimes \a}(X;\rnt) := H_*(C_*^{\phi\otimes \a}(X;\rnt)).
$$
If $\a$ and $\phi$ are understood we will drop them from the notation.
 Since $X$ is compact and since $\rt$
is Noetherian these modules are finitely presented over $\rt$.
We now define the \emph{twisted Alexander polynomial of $(X,\phi,\a)$} to be the order of $H_1(X;\rnt)$ (see \cite{FV10a} and \cite[Section~4]{Tu01} for details). We will denote it as $\Delta_{X,\phi}^\a\in \rt$.
Note that $\Delta_{X,\phi}^\a\in \zt$  is well-defined up to multiplication by a unit in $\rt$.
We adopt the convention that we drop $\a$ from the notation if $\a$ is the trivial representation to $\gl(1,\Z)$.

If $\a\colon \pi_1(N)\to G$ is a homomorphism to a finite group $G$, then we get the regular representation $\pi_1(N)\to G\to \aut_\Z(\Z[G])$, where the second map is given by left multiplication. We can identify $\aut_\Z(\Z[G])$ with $\gl(|G|,\Z)$ and we obtain the corresponding twisted Alexander polynomial $\Delta_{N,\phi}^\a$.
We will sometimes write $H_*(X;\zgt)$ instead of $H_*(X;\Z^{|G|}\tpm)$.

The following lemma is well-known (see e.g. \cite[Remark~4.5]{Tu01}).

\begin{lemma}\label{lem:deltazero}
Let $(N,\phi)$ be a 3--manifold pair
and let  $\a\colon \pi_1(N)\to G$ be a homomorphism to a finite group. Then $\Delta_{N,\phi}^\a\ne 0$ if and only if $H_1(N;\zgt)$ is $\zt$-torsion.
\end{lemma}

Later, we will need the following well-known lemma.

\begin{lemma}\label{lem:alsozero} Let $(N,\phi)$ be a 3--manifold pair.
Let  $\a\colon \pi_1(N)\to G$ and $\b\colon \pi_1(N)\to H$ be homomorphisms to finite groups such that $\ker(\a)\subset \ker(\b)$.
Then there exists $p\in \qt$ such that
\[ \Delta_{N,\phi}^\a=\Delta_{N,\phi}^\b\cdot p \in \qt.\]
In particular, if $\Delta_{N,\phi}^\b=0$, then $\Delta_{N,\phi}^\a=0$.
\end{lemma}

\begin{proof}
We denote by $\a$ also  the regular representation $\pi_1(N)\to \aut_\Q(\Q[G])$, and similarly we denote by $\b$ the regular representation $\pi_1(N)\to \aut_\Q(\Q[H])$. Note that the assumption  $\ker(\a)\subset \ker(\b)$
implies there exists an epimorphism $\g\colon G\to H$ such that $\b=\g\circ \a$.
Note that $\g$ endows $\Q[H]$ with the structure of a left $\Q[G]$-module. It follows from Maschke's theorem
(see \cite[Theorem~XVIII.1.2]{La02}) that there exists a left $\Q[G]$-module $P$ and an isomorphism of
left $\Q[G]$-modules
\[ \Q[G] \cong \Q[H]\oplus P.\]
We now denote by $\rho$ the representation $\pi_1(N)\xrightarrow{\a}G\to \aut(P)\cong \gl(\dim(P),\Q)$.
It now follows from the definitions that
\[ \Delta_{N,\phi}^\a=\Delta_{N,\phi}^\b\cdot \Delta_{N,\phi}^\rho.\]
\end{proof}

\section{Twisted Alexander polynomials and fibered 3--manifolds}

Let $(N,\phi)$ be a 3--manifold pair.
We  denote by $\|\phi\|_T$ the Thurston norm of a class $\phi \in H^{1}(N;\Z)$: we refer to \cite{Th86} for details.
We say that $p(t)\in \zt$ is \emph{monic} if its top coefficient equals $\pm 1$ and  given a nonzero polynomial $p(t)\in \zt$ with $p=\sum_{i=k}^l a_it^i, a_k\ne 0, a_l\ne 0$ we define $\deg(p)=l-k$.

It is known that twisted Alexander polynomials give complete fibering obstructions. In fact the following theorem holds:

\begin{theorem}\label{thm:fv08}  Let $(N,\phi)$ be a $3$--manifold pair where $N \neq S^1 \times S^2, S^1 \times D^2$.
 Then $(N,\phi)$ is fibered if and only if
 for
 any epimorphism $\a\colon \pi_1(N)\to G$ onto a finite group
the twisted Alexander polynomial $\Delta_{N,\phi}^{\a}\in \zt$ is monic
and
\[ \deg(\Delta_{N,\phi}^{\a})= |G| \, \cdot \|\phi\|_T + (1+b_3(N)) \div \, \phi_{\a},\]
where $\phi_\a$ denotes the restriction of $\phi\colon \pi_1(N)\to \Z$ to $\ker(\a)$, and where
we denote by  $\div \,\phi_{\a} \in \N$ the divisibility of $\phi_{\a}$, i.e.
\[ \div \, \phi_{\a} =\max\{ n\in \N \, |\, \phi_{\a}=n\psi \mbox{ for some }\psi\colon \ker(\a)\to \Z\}.\]
 \end{theorem}

The `only if' direction has been shown at various levels of generality by
Cha \cite{Ch03}, Kitano and Morifuji \cite{KM05}, Goda, Kitano and Morifuji \cite{GKM05}, Pajitnov \cite{Pa07}, Kitayama \cite{Kiy07}, \cite{FK06} and \cite[Theorem~6.2]{FV10a}.
The `if' direction is the main result of \cite{FV11a}. We also refer to \cite{FV11b} for a more leisurely approach to the proof of the `if' direction.

\section{The proof of Theorem \ref{mainthmfibintro}} \label{sec:proof}

In this section we will prove Theorem \ref{mainthmfibintro}.
The approach we follow is the one we used in \cite{FV08b} to cover the case of a $3$--manifold with certain subgroup separability properties, adding as new ingredient the work of Wilton and Zalesskii \cite{WZ10} (which builds in turn on work of  Hamilton \cite{Ham01}), to cover the case where the $3$--manifold has a nontrivial Jaco--Shalen--Johannson (JSJ) decomposition
(see the original work of Jaco--Shalen \cite{JS79} or Johannson \cite{Jo79} or also \cite{AFW12} for details).
We start with the following.

\begin{lemma}\label{thm:extendhom}
Let $N$ be an irreducible 3-manifold with JSJ pieces $N_v, v\in V$. Let $\a_v\colon \pi_1(N_v)\to G_v, v\in V$ be homomorphisms to finite groups.
Then there exists an epimorphism $\b\colon \pi_1(N)\to G$ to a finite group, such that $\ker(\b)\cap \pi_1(N_v)\subset \ker(\a_v)$ for all $v\in V$.
\end{lemma}

\begin{proof}
We write $K_v:=\ker(\a_v), v\in V$. Evidently there exists an $n\in \N$ with the following property:
for each $v\in V$  and each boundary torus $T$ of $N_v$ we have
\[ n\cdot \pi_1(T)\subset \pi_1(T)\cap K_v.\]
(Here $n\cdot \pi_1(T)$ denotes the unique subgroup of $\pi_1(T)\cong \Z^2$ of index $n^2$.)
By \cite[Theorem~3.2]{WZ10} (which relies strongly on Lemmas 5 and 6 of \cite{Ham01}) there exists an $m\in \N$ and finite index normal subgroups
$L_v\subset \pi_1(N_v), v\in V$ such that for each $v\in V$ and each boundary torus $T$ of $N_v$ we have
\[ nm\cdot \pi_1(T)=\pi_1(T)\cap L_v.\]
We now define $M_v=K_v\cap L_v, v\in V$. Note that for each JSJ torus $T$ and two adjacent JSJ pieces the intersection of the corresponding $M$-groups is  $nm\cdot \pi_1(T)$. It now follows from a standard argument (see e.g. \cite[Proof~of~Theorem~3.7]{WZ10}) that there exists a finite index normal subgroup $M$ of
$\pi_1(N)$ such that $M\cap \pi_1(N_v)\subset M_v$ for any $v\in V$. Clearly the epimorphism $\pi_1(N)\to \pi_1(N)/M$ has the desired properties.
\end{proof}

For the reader's convenience we recall the statement of Theorem \ref{mainthmfibintro}.

\begin{theorem}Let $(N,\phi)$ be a 3--manifold pair.
Then if $\phi \in H^1(N)$ is nonfibered, there exists an epimorphism $\a\colon \pi_1(N)\to G$ onto a finite group $G$ such that
\[ \Delta_{N,\phi}^\a =0.\]
\end{theorem}

\begin{remark} As mentioned above, a proof of this theorem appears in \cite{FV08b} for  a $3$--manifold $N$ in the following two cases:
\bn
\item if the subgroups carried by Thurston norm minimizing surfaces are separable, or
\item if $N$ is a graph manifold.
\en
Recent work of Przytycki and Wise \cite[Theorem~1.1]{PW11} shows that the separability condition (1) is in fact satisfied for all graph manifolds, which gives an alternative proof for (2).
\end{remark}

\begin{proof} If $N$ is reducible, the statement is proven to hold in \cite[Lemma~7.1]{FV11a}. Hence we restrict ourselves to the case where $N$ is irreducible.
We denote by $\{N_{v}\}_{v\in V}$ the set of JSJ components  of $N$ and by $\{T_e\}_{e\in E}$ the set of JSJ tori in the JSJ decomposition of $N$.
Let  $\phi\in H^1(N;\Z)$ be  a nonfibered class.
If the restriction of $\phi$ to one of the JSJ tori is trivial, then it follows from \cite[Theorem~5.2]{FV08b}
that there exists an epimorphism $\a\colon \pi_1(N)\to G$ onto a finite group $G$ such that $\Delta_{N,\phi}^\a =0$.
We will henceforth assume that the restriction of $\phi$ to all JSJ tori is nontrivial.

Given $v\in V$ we denote by $\phi_v$ the restriction of $\phi$ to $N_v$.
Since $\phi$ is nonfibered it follows from   \cite[Theorem~4.2]{EN85} that there exists a $w\in V$, such that
$(N_w,\phi_w)$ is not fibered. If $N_{w} = N$ is hyperbolic and closed  (i.e. the JSJ decomposition  is trivial), as we assume that $(N,\phi)$ is not fibered, $N$ admits an incompressible surface $\Sigma$ which does not lift to a fiber in any finite cover. By a result of Bonahon and Thurston (see \cite{Bon86}) the surface $\Sigma$ is a geometrically finite surface. It then follows from Theorem \ref{thm:wiseintro} that $\pi_1(N)$ is LERF.
If $N_w$ is hyperbolic and has nontrivial boundary, Theorem \ref{thm:wiseintro} guarantes again that $\pi_1(N_w)$ is LERF. The same holds if $N_w$ is Seifert fibered, by Scott's theorem  (see \cite{Sc78}).

By \cite[Theorem~4.2]{FV08b} there exists in either case an epimorphism $\a_w\colon \pi_1(N_w)\to G_w$ onto a finite group $G_w$ such that $\Delta_{N_w,\phi_w}^{\a_w}=0$.
By Lemma \ref{thm:extendhom} above
there exists  a homomorphism $\b\colon \pi_1(N)\to G$ to a finite group, such that $\ker(\b)\cap \pi_1(N_w)=\ker(\b_w) \subset \ker(\a_w)$.
(Here, given $v\in V$ we denote by $\b_v$ the restriction of $\b$ to $N_v$.)
By Lemma \ref{lem:alsozero} we also have $\Delta_{N_w,\phi_w}^{\b_w}=0$.

Now, there exists a Mayer--Vietoris type long exact sequence of twisted homology groups:
\[ \dots \to \bigoplus_{e\in E} H_i(T_e;\zgt)\to \bigoplus_{v\in V} H_i(N_v;\zgt)\to H_i(N;\zgt)\to \dots.\]
Recall that  we assumed that the restriction of $\phi$ to $T_e$ is nontrivial for each $e$. It follows from a straightforward calculation (see e.g. \cite[p.~644]{KL99}) that $H_i(T_e;\zgt)$ is $\zt$-torsion for each $i$ and each $e\in E$.
As  $\Delta_{N_w,\phi_w}^{\b_v}=0$, Lemma \ref{lem:deltazero} implies that  $H_1(N_w;\zgt)$ is not $\zt$-torsion.
But then it follows from the above exact sequence that $H_1(N;\zgt)$ is not $\zt$-torsion either, i.e. $\Delta_{N,\phi}^\b=0$.
\end{proof}

\section{Applications to 3--manifold topology}\label{section:firstapps}

Let $K\subset S^3$ be a knot. We denote by $X_K:=S^3\sm \nu K$ the exterior of $K$.
A regular Morse function on $X_K$ is a function $f\colon X_K\to S^1$ such that all singularities are nondegenerate and which restricts on the boundary of $X_K$ to a fibration
with connected fiber. Given a Morse map $f$ we denote by $m_i(f)$ the number of critical points of index $i$.
A regular Morse function $f\colon X_K\to S^1$ is called \emph{minimal} if $m_i(f)\leq m_i(g)$ for any regular Morse map $g$ and any $i$.
It is shown by  Pajitnov, Rudolph and Weber \cite{PRW02} that any knot admits a minimal regular Morse function.
Its number of critical points is called the \emph{Morse-Novikov number of $K$} and denoted by $\MN(K)$.
Note that $K$ is fibered if and only if $\MN(K)=0$. It is   known that $\MN(K_1\# K_2)\leq \MN(K_1)+\MN(K_2)$ (see \cite[Proposition~6.2]{PRW02}), but it is not known whether equality holds. It is not even known whether $\MN(n\cdot K)=n\cdot \MN(K)$.
The following theorem can be viewed as evidence to an affirmative answer for the latter question.

\begin{theorem}
Let $K\subset S^3$ be a nonfibered knot. Then there exists a $\l>0$, such that
\[ \MN(n\cdot K)\geq n\cdot \l.\]
\end{theorem}

This theorem is an immediate consequence of Theorem \ref{mainthmfibintro} and  results of Goda and Pajitnov \cite[Theorem~4.2~and~Corollary~4.6]{GP05}.
The statement is similar in spirit to a result by Pajitnov (see \cite[Proposition~4.2]{Pa10}) on the tunnel number of multiples of a given knot.

Now let $(N,\phi)$ be a fibered 3--manifold pair. In that case $\ker(\phi\colon \pi_1(N)\to \Z)$ is the fundamental group of a surface, in particular it is a finitely generated group, it follows that
 the $\phi$-cover of $N$ admits only countably many finite covers.
 The following theorem, which follows from combining Theorem \ref{mainthmfibintro} with work of Silver and Williams (see \cite{SW09a,SW09b}), says that the converse to the above statement holds true.

\begin{theorem}\label{thm:fibcovers}
Let $N$  be a 3-manifold  and let $\phi\in H^1(N)=\hom(\pi_1(N),\Z)$.
If $(N,\phi)$ does not fiber over $S^1$, then the $\phi$-cover of $N$ admits uncountably many finite covers.
\end{theorem}

For knot exteriors this result is an immediate consequence of Theorem \ref{mainthmfibintro}  and
a result of Silver and Williams \cite[Theorem~3.4]{SW09b} for knots (see also \cite{SW09a}). It is straightforward to verify that the argument by Silver and Williams  carries over to the general case.

Note that Theorem \ref{thm:fibcovers} can be viewed as a significant strengthening of Stallings' fibering theorem (see \cite{St62}), which says that a class $\phi\in H^1(N)=\hom(\pi_1(N),\Z)$
is fibered if and only if $\ker(\phi)$ is finitely generated.

\section{Symplectic  $4$--manifolds with a free circle action} \label{section:symp}

In this section we will prove Theorem \ref{thm:symp}.

\subsection{Preliminaries} \label{algtop}
\label{section2} We start by recalling some elementary facts about the algebraic topology of a $4$--manifold $M$ that carries a free circle action. The free circle action renders $M$ the total space of a circle bundle $p \colon M \to N$ over the orbit space, with Euler class $e \in H^2(N)$. For the purpose of proving Theorem \ref{thm:symp}, we will see that we can limit the discussion to the case where the Euler class is nontorsion, so we will make this assumption for the rest of the subsection.
The Gysin sequence reads
\[ \xymatrix{  H^{0}(N) \ar[d]^\cong \ar[r]^{\cup \hspace*{1pt} e}& H^2(N) \ar[d]^\cong \ar[r]^{p^{*}}&
H^2(M) \ar[d]^\cong \ar[r]^{p_{*}} &H^{1}(N) \ar[d]^\cong \ar[r]^{\cup \hspace*{1pt} e} &H^3(N)\ar[d]^\cong \\
H_3(N)\ar[r]^{\cap e}&H_1(N)\ar[r]&H_2(M)\ar[r]^{p_*}&H_2(N)\ar[r]^{\cap e}&H_0(N),}
\]
where $p_* \colon H^{2}(M) \to H^{1}(N)$ denotes integration along the fiber.
In particular we have \[ 0 \to
\langle e \rangle \to H^2(N) \xrightarrow{p^{*}} H^2(M) \xrightarrow{p_{*}} \mbox{ker}( \cdot e)
\to 0, \] where $\langle e \rangle$ is the cyclic subgroup of $H^2(N)$
generated by the Euler class and where $\mbox{ker}( \cdot e)$ denotes the subgroup of  elements in $H^1(N)$ whose pairing
with the Euler class vanishes. As $e$ is nontorsion, it follows that
$b_2(M)=2b_1(N) - 2$.
 It is not difficult to verify that $\operatorname{sign}(M) = 0$, hence $b_2^+(M)=b_1(N) - 1$.

Let $\a\colon \pi_1(N)\to G$ be a homomorphism to a finite group. We denote by $\pi\colon  N_{G} \to N$ the regular $G$--cover of $N$. It is well known that $b_1(N_{G}) \geq b_1(N)$.  If $\pi\colon M\to N$ is a circle bundle with Euler class $e\in H^2(N)$
 then $\a$ determines a regular $G$--cover of $M$ that we will
denote (with slight abuse of notation) $\pi\colon  M_{G} \to M$. These covers are related by the
commutative diagram
\begin{equation}
\xymatrix{  M_G \ar[r]^{\pi} \ar[d] & M\ar[d] \\
N_G \ar[r]^{\pi}  & N}
\ee
where the circle bundle
 $p_{G} \colon  M_{G} \to N_{G}$ has Euler class $e_{G} = \pi^{*} e \in
H^{2}(N_G)$ that is nontorsion as $e$ is.

\subsection{Seiberg-Witten theory for symplectic manifolds with a circle action}

In this section we will apply, for the class of manifolds we are studying, Taubes' nonvanishing theorem for Seiberg--Witten invariants of symplectic manifold to get a restriction on the class of orbit spaces of a free circle action over a symplectic $4$--manifold. In order to do so, we need to understand the Seiberg-Witten invariants of
$M$. Again, we will limit the discussion here to the case where the Euler class $e \in H^2(N)$ is not torsion. (The torsion case will be treated as a corollary of \cite{FV11a}.)
The essential ingredient is the fact that the Seiberg-Witten invariants of $M$ are related to the Alexander polynomial of $N$. Baldridge proved the following result, that combines Corollaries 25 and 27 of \cite{Ba03} (cf. also \cite{Ba01}), to which we refer the reader for definitions and results for Seiberg-Witten theory in this set-up:

\begin{theorem} \label{bald} \label{thm:baldridge} \textbf{\emph{(Baldridge)}} Let $M$ be a $4$--manifold admitting a free circle action with nontorsion Euler class $e \in H^2(N)$, where $N$ is the orbit
space. Then the Seiberg-Witten invariant $SW_{M}(\kappa)$ of a class $\kappa = p^{*} \xi \in p^*
H^{2}(N) \subset H^{2}(M)$ is given by the formula \begin{equation} \label{baldfor} SW_{M}(\kappa)
= \sum_{l \in \Z} SW_{N}(\xi+le) \in \Z,  \end{equation} in particular when $b_{2}^{+}(M) = 1$ it is independent of the chamber in which it was calculated. Moreover, if $b_{2}^{+}(M) > 1$, these are the only basic classes.
\end{theorem}

As  $H^2(N)$ acts freely and transitively on $\spin^c(N)$  (and similarly for $M$), we use the existence of a product spin$^c$ structure on $N$ (respectively $M$) to henceforth identify the set of spin$^c$ structures with $H^2(N)$ (respectively $H^2(M)$). Observe that, with this identification, the first Chern class of a spin$^c$ structure on $N$ satisfies $c_1(\xi) = 2 \xi \in H^2(N)$ (and similarly on $M$.)

The Seiberg--Witten invariants of $N$ determine, via \cite{MT96}, the Alexander polynomial of $N$. Assuming that a manifold $M$ as above is symplectic, we will use Taubes' constraints on its Seiberg-Witten invariants and Baldridge's formula to get a constraint on the twisted Alexander polynomials of $N$. We start with a technical lemma.

\begin{lemma} \label{prop:pullback} Let $(M,\omega)$ be a symplectic manifold admitting a free circle
action with nontorsion Euler class $e \in H^2(N)$, where $N$ is the orbit space. Then the
canonical class $K \in H^2(M)$ of the symplectic structure is the pull-back of a class $\zeta \in
H^2(N)$, where $\zeta$ is well--defined up to the addition of a multiple of $e$.\end{lemma}

\begin{proof}
If $b^{+}_{2}(M) > 1$ this is a straightforward consequence of Theorem \ref{bald}, as the canonical
class by \cite{Ta94} is a basic class of $M$, hence must be the pull-back of a class of $H^2(N)$. The case
of $b_{2}^{+}(M) = 1$ can be similarly obtained by a careful analysis of the chamber structure of the Seiberg-Witten invariants for classes that are not pulled back from $N$, but it is possible to use a quicker argument. First, observe that as the $2$-torsion of $H^2(M)$ is contained in $p^{*} H^2(N)$, a spin$^c$ structure on $M$ is pull-back if and only if its first Chern class is pull--back. Next,  starting from a closed curve in
$N$ representing a suitable element of $H_1(N)$, we can identify a torus $T \subset M$ of
self--intersection zero, representing the generator of a cyclic subgroup in the image of the map
$H_1(N) \to H_2(M)$ in the homology Gysin sequence, that satisfies $\omega \cdot [T] > 0$.  Second
we can assume, by \cite{Liu96}, that $K \cdot \omega \geq 0$. (Otherwise $M$ would be a rational
or ruled surface; using classical invariants, the only possibility would be $M = S^2 \times T^2$, but then $e = 0$.)  Also, as both signature and
Euler characteristic of $M$ vanish, $K^2 = 2 \chi(M) + 3 \sigma(M) = 0$. Omitting the case of $K$
torsion, where the statement is immediate, we deduce that both $K$ and (the Poincar\'e dual of)
$[T]$ lie in the closure of the forward positive cone in $H^2(M,\R)$ determined by $\omega$. The
light-cone Lemma (see e.g. \cite{Liu96}) asserts at this point that $K \cdot [T] \geq 0$. On the
other hand, the adjunction inequality of Li and Liu (see \cite{LL95}) gives $K \cdot [T] \leq  0$,
hence $K \cdot [T] = 0$. It now follows that $K$ is a multiple of $PD([T])$, in particular the
pull-back of a class on $N$.
\end{proof}

Let then $(M,\omega)$ be a symplectic manifold admitting a free circle action with orbit space $N$ and nontorsion Euler
class $e \in H^{2}(N)$. As $[\omega]^2 > 0$ it follows from Section \ref{algtop} that $[\omega] \notin
p^*(H^{2}(N,\R))$, and in particular $p_{*}[\omega] \neq 0 \in H^{1}(N;\R)$. Using openness of the
symplectic condition, we can assume that $[\omega] \in H^2(M;\R)$ lies in the rational lattice
(identified with) $H^2(M;\Q)$. After suitably scaling $\omega$ by a rational number if needed, the
class $p_{*}[\omega]$ is then (the image of) a primitive (in particular, nonzero) class in
$H^1(N;\Z)$ that we denote by $\phi$.

We are in position now to use Equation (\ref{baldfor}) to obtain the following.

\begin{proposition} \label{monicness} Let $(M,\omega)$ be a symplectic manifold admitting a free circle
action with nontorsion Euler class such that $\phi = p_{*}[\omega] \in H^{1}(N)$ is an integral class on the orbit space $N$.
 Then for all epimorphisms $\alpha \colon  \pi_1(N) \to G$ to a finite group the twisted Alexander polynomial
 $\Delta_{N,\phi}^{\alpha} \in \Z[t^{\pm 1}]$ is monic of Laurent degree
 \[ \deg \, \Delta_{N,\phi}^{\alpha} = |G| \, \zeta \cdot \phi + 2 \div \, \phi_{G}.\]
 Here,  $\zeta \in H^2(N)$ is a class whose pull--back to $M$ gives the canonical
 class of $M$. Furthermore $\phi_{G}$ denotes the restriction of
 $\phi \colon  \pi_1(N) \to \Z$ to $Ker \a$ and $\div \, \phi_G$ stands for the divisibility of $\phi_G$. \end{proposition}

\begin{proof} The proof follows by application of Taubes' results (\cite{Ta94,Ta95}) on the Seiberg--Witten invariants of finite covers of $M$ to impose constraints on
the twisted Alexander polynomials of $N$. We will first analyze the constraints on the ordinary
$1$--variable Alexander polynomial $\Delta_{N,\phi}$. By \cite[Theorem~3.5]{FV08a} we can write this polynomial
as \be \label{equ:onemulti} \Delta_{N,\phi} = (t^{div {\phi}} - 1)^2 \cdot \sum_{g \in H} a_{g}
t^{\phi (g)} \in \Z[t^{\pm 1}], \ee where $H$ is the maximal free abelian quotient of $\pi_{1}(N)$
and $\Delta_{N} = \sum_{g \in H} a_{g} \cdot g \in \Z[H]$ is the ordinary multivariable Alexander
polynomial of $N$. By Meng and Taubes (\cite{MT96}) the latter is related to the Seiberg--Witten
invariants of $N$ via the formula
\begin{equation} \label{mengtaubes} \sum_{g \in H} a_{g}
\cdot g = \pm \sum_{\xi \in H^{2}(N)} SW_{N}(\xi) \cdot  f(\xi) \in \Z[H],
\end{equation}
where $f$ denotes the composition of Poincar\'e duality with the quotient map $f\colon  H^2(N) \cong
H_{1}(N) \to H$.
Using this formula, we can write
\be \label{inter} \Delta_{N,\phi} = \pm (t^{div {\phi}} - 1)^2 \sum_{\xi\in H^2(N)} SW_N(\xi)t^{\phi \cdot \xi}. \ee
We will use now Equation (\ref{baldfor}) to write $\Delta_{N,\phi}$ in terms of the $4$--dimensional Seiberg-Witten invariants of $M$.
In order to do so, observe that for all classes $\xi \in H^2(N)$ we can write $\xi \cdot \phi = \xi \cdot p_{*} \omega = p^{*} \xi \cdot \omega = \kappa \cdot \omega$ where $\kappa = p^{*} \xi$. Grouping together the contributions of the $3$--dimensional basic classes in terms of their image in $H^2(M)$, and using the fact that $(\xi + l e) \cdot \phi = \xi \cdot\phi$ and  (\ref{baldfor}) we get
\[ \ba{rcl} \Delta_{N,\phi} & = & \pm (t^{div \, \phi}-1)^2 \sum\limits_{\kappa \in p^{*} H^2(N)} \sum\limits_{l} SW_N(\xi + le) t^{\phi \cdot \xi} \\ \\ & = & \pm (t^{div \, \phi}-1)^2 \sum\limits_{\kappa \in p^{*} H^2(N)} SW_M(\kappa)t^{\kappa \cdot
\omega }.\ea \]

Taubes' constraints, applied to the symplectic manifold $(M,\omega)$, assert that the canonical spin$^c$ structure $\kappa_{\omega}$, with first Chern class $c_1(\kappa_{\omega}) = - K = - p^{*} \zeta \in H^{2}(M)$, satisfies $SW_{M}(\kappa_{\omega}) = 1$. Moreover, for all spin$^c$ structure $\kappa$ with $SW_{M}(\kappa) \neq 0$, we have \begin{equation}
\label{more} \kappa_{\omega} \cdot \omega \leq \kappa \cdot \omega, \end{equation} with
equality possible only for $\kappa = \kappa_{\omega}$. (When $b^{+}_{2}(M) = 1$, this statement applies to the Seiberg-Witten invariants evaluated in Taubes' chamber, but as remarked in Theorem \ref{bald} this specification is not a concern in our situation.)
 It now follows that $\Delta_{N,\phi}$ is a monic
polynomial, and remembering the symmetry of $SW_N$ (or $\Delta_{N,\phi}$), we see that its Laurent degree is $d =-2 \kappa_{\omega}\cdot \omega  + 2 \div \phi =  K \cdot \omega + 2 \div \phi = \zeta \cdot \phi + 2 \div \phi$.

Consider now any symplectic $4$--manifold $M$ satisfying the hypothesis of the statement. Take an
epimorphism $\a \colon  \pi_1(N) \to G$ and denote, as in Section \ref{algtop}, by $N_{G}$ and $M_{G}$ the associated $G$--covers
of $N$ and $M$ respectively.
We will bootstrap the constraint on the ordinary Alexander polynomials to the
twisted Alexander polynomials associated to $\a \colon  \pi_1(N) \to G$. As $(M,\omega)$ is symplectic, $M_{G}$ inherits a symplectic form $\omega_{G} :=
\pi^{*}\omega$ with canonical class $K_G := \pi^{*} K$, that is easily shown to satisfy the
condition $\phi_{G} := (p_{G})_{*}[\omega_{G}] = \pi^{*}\phi \in H^{1}(N_{G})$. We can therefore apply the results
of the previous paragraph to the pair
$(N_G,\phi_G)$ to get a constraint for $\Delta_{N_G,\phi_G}$. This, together with the relation
$\Delta_{N,\phi}^{\a} = \Delta_{N_G,\phi_G}$ proven in \cite[Lemma~3.3]{FV08a} and some straightforward
calculations show that $\Delta_{N,\phi}^{\a}$ is monic of the degree stated. \end{proof}

Note that the adjunction inequality implies that $\zeta \cdot \phi \leq \| \phi \|_{T}$.
 If we had a way to show that this is an equality, then all twisted Alexander polynomials would have maximal degree. Theorem \ref{thm:fv08} would then suffice to show that  $\phi \in H^{1}(N)$ is fibered. However (except in the case of $S^1 \times N$ where we can use Kronheimer's theorem, see \cite{Kr99} and \cite{FV08a}) we are not aware of any direct way to show   that  $\zeta \cdot \phi = \| \phi \|_{T}$, which compels us to use the  stronger Theorem  \ref{mainthmfibintro}.

\subsection{Proof of Theorem \ref{thm:symp}}

For the reader's convenience we recall the statement of Theorem \ref{thm:symp}.

\begin{theorem} \label{thm:123}Let $N$ be a closed 3-manifold and let $p\colon M\to N$ be an $S^1$--bundle over $N$. We denote by $p_*\colon H^2(M;\R)\to H^1(N;\R)$ the map which is given by integration along the fiber.
Let $\psi\in H^2(M;\R)$.The following are equivalent:
\bn
\item $\psi$ can be represented by a symplectic structure,
\item $\psi$ can be represented by a symplectic structure which is $S^1$--invariant,
\item $\psi^2>0$ and  $\phi=p_*(\psi) \in H^1(N;\R)$  lies in the open cone on a fibered face of the Thurston norm ball of $N$.
\en
\end{theorem}

\begin{proof}
The statement `(2) implies (1)' is trivial and the statement  `(3) implies (2)' is proved in \cite{FV10b}, generalizing earlier work of Thurston \cite{Th76}, Bouyakoub \cite{Bo88}
and Fern\'andez, Gray and Morgan \cite{FGM91}.

We now turn to the proof of `(1) implies (3)'.
If  $\psi\in H^2(M;\R)$ can be represented by a symplectic structure then it follows from the definition of being symplectic that $\psi^2>0$.

Note that by \cite[Theorem~5]{Th86} a class $\phi\in H^1(N;\R)$  lies in the open cone on a fibered face of the Thurston norm ball of $N$
if and only if $\phi$ can be represented by a non-degenerate closed 1-form. It now follows from \cite[Proposition~3.1]{FV10b} that it suffices to
prove `(1) implies (3)' for classes $\psi$ such that $p_*([\w])$ is an integral class.

 We first assume that the Euler class of the $S^1$-bundle is nontorsion.
Suppose that we have $\psi\in H^2(M;\R)$ which
can be represented by a symplectic structure  such that $p_{*} \psi \in H^{1}(N)$ is an integral class.  It follows immediately
from Proposition \ref{monicness} that $\Delta_{N,p_*(\psi)}^\a$ is nonzero for any  epimorphism $\alpha \colon  \pi_1(N) \to G$ onto a finite group.
Therefore, as a consequence of Theorem \ref{mainthmfibintro} it follows that $p_*(\psi)$ is a fibered class.

We now assume that the Euler class of the $S^1$-bundle is trivial, i.e. $M=S^1\times N$. This case has been completely solved in \cite{FV11a}. (More in the spirit of the present paper one can follow the argument for the case of nontorsion Euler class, replacing Proposition \ref{monicness}
by \cite[Proposition~4.4]{FV08a}.)

Finally, if the Euler class of the $S^1$-bundle is torsion, then it is well-known that there exists a finite cover of $N$ such that the pull-back $S^1$-bundle has trivial Euler class.
We refer to \cite[Proposition~3]{Bow09} and \cite[Theorem~2.2]{FV11c} for details.
Note that it is a consequence of Stallings' fibering theorem (see \cite{St62}) that an integral class $\phi \in H^1(N;\Z)$ is fibered if and only if the pull back to a finite cover
is fibered. The case of torsion Euler class now follows easily from this observation and from the product case. We leave the details to the reader.
\end{proof}


\end{document}